\documentclass[11pt]{article}
\usepackage{amssymb}

\usepackage{amsmath}

\newtheorem{theorem}{Theorem}

\newtheorem{definition}[theorem]{Definition}

\newtheorem{lemma}[theorem]{Lemma}

\newtheorem{proposition}[theorem]{Proposition}
\newtheorem{remark}[theorem]{Remark}

\newenvironment{proof}[1][Proof]{\textbf{#1.} }{\ \rule{0.5em}{0.5em}}
\newdimen\dummy
\dummy=\oddsidemargin
\begin{document}

\title{Ergodicity and Gaussianity for Spherical Random Fields\thanks{%
We are grateful to Mirko D'Ovidio for some useful comments on an earlier
version.}}
\author{Domenico Marinucci and Giovanni Peccati \\
{\small Dipartimento di Matematica, Universit\`{a} di Roma Tor Vergata} \\
{\small Laboratoire MODAL\ X, Universit\`{e} de Paris X}}
\maketitle

\begin{abstract}
We investigate the relationship between ergodicity and asymptotic
Gaussianity of isotropic spherical random fields, in the high-resolution (or
high-frequency) limit. In particular, our results suggest that under a wide
variety of circumstances the two conditions are equivalent, i.e. the sample
angular power spectrum may converge to the population value if and only if
the underlying field is asymptotically Gaussian, in the high frequency
sense. These findings may shed some light on the role of \textit{Cosmic
Variance }in Cosmic Microwave Background (CMB) radiation data analysis.

\begin{itemize}
\item Keywords and Phrases: Spherical Random Fields, High-Frequency
Gaussianity, High-Frequency Ergodicity, Gaussian Subordination,
Clebsch-Gordan Coefficients, Cosmic Variance

\item PACS: 02.50-r, 98.70-Vc, 98.80-k

\item AMS\ Classification: 60G60; 60G15, 60K40, 62M15, 42C10
\end{itemize}
\end{abstract}

\section{Introduction and background}

\subsection{Overview}

The usual framework for proving asymptotic results in probability (for
instance, central limit theorems or laws of large numbers) lies within the
so-called \textsl{large sample paradigm}, according to which more and more
(independent or weakly dependent) random variables are generated, and the
limiting behaviour of some functionals of these variables (e.g., averages or
empirical moments) is studied.

Physical applications, however, are prompting the development of a
stochastic asymptotic theory of a rather different nature, where the
indefinite repetition of a single experience is no longer available, and one
relies instead on observations of the same (fixed) phenomenon with higher
and higher degrees of resolution.

One crucial instance of this situation appears when dealing with the
statistical analysis of random fields indexed by compact manifolds, the
quintessential example being provided by the case of the sphere $S^{2}$.
Indeed, we are especially concerned with issues arising from the analysis of
the Cosmic Microwave Background (CMB) radiation, a theme which is currently
at the core of physical and cosmological research, see for instance \cite%
{DOD, KOT} for textbook references, and \cite{WMAP1,WMAP2,WMAP3} for further
discussions around the latest experimental data.

It is well-known that the CMB\ is a relic electromagnetic radiation
providing a snapshot of the Universe at the so-called \textsl{age of
recombination}, i.e. at the era when electrons in the primordial fluid
arising from the Big Bang were captured by protons to form stable hydrogen
atoms. Since the cross-section of hydrogen atoms is much smaller than for
free electrons, after recombination photons can be viewed as diffusing
freely across the Universe (to first order approximations). According to the
latest experimental evidence, this has occurred some $3.7\times 10^{5}$
years after the Big Bang, i.e. 13.7 billion years from the current epoch.
Several experiments have been devoted to collecting extremely refined
observations of the CMB, the leading role being played by the currently
ongoing NASA mission WMAP (launched in 2001, see \texttt{%
http://map.gsfc.nasa.gov/}) and the ESA mission Planck, which is just now
starting to operate after the launch on May 14, 2009 (see \texttt{%
http://www.sciops.esa.int/}).

From a mathematical point of view, the CMB can be regarded as a single
realization of an isotropic, zero-mean, finite variance spherical random
field, for which the following spectral representation holds (see e.g. \cite%
{ADT} or \cite{LEO})%
\begin{equation}
T(x)=\sum_{l=0}^{\infty }\sum_{m=-l}^{l}a_{lm}Y_{lm}(x)\text{ , }x\in S^{2}%
\text{ .}  \label{basfor}
\end{equation}%
Here, the collection%
\begin{equation*}
\left\{ Y_{lm}:l\geq 0,\text{\ }m=-l,...,l\right\}
\end{equation*}%
stands for the usual triangular array of spherical harmonics, which are
well-known to provide a complete orthonormal system for the $L^{2}(S^{2})$
space of square-integrable functions (with respect to Lebesgue measure)\ on
the sphere -- see \cite{STW,VMK,VIK}. In a loose sense, we can say that the
frequency parameter $l$ is related to a characteristic angular scale, say $%
\vartheta _{l}$, according to the relationship $\vartheta _{l}\simeq \pi /l$%
. The (random) triangular array of spherical harmonic coefficients $\left\{
a_{lm}:l\geq 0,\text{ \ }m=-l,...,l\right\} $ are such that $Ea_{lm}=0$ and $%
Ea_{lm}\overline{a}_{l^{\prime }m^{\prime }}=C_{l}\delta _{l}^{l^{\prime
}}\delta _{m}^{m^{\prime }}$, the bar denoting complex-conjugation and $%
\delta _{a}^{b}$ indicating the Kronecker delta function. The non-negative
sequence $\left\{ C_{l}:l\geq 0\right\} $ (not depending on $m$ -- see \cite%
{MPJMVA} as well as the forthcoming section) is the \textsl{angular power
spectrum} of the spherical field (see for instance \cite{BaMa, BaMaVa}).

As recalled above, our work deals with asymptotic issues, where the
expression ``asymptotic''\ has to be understood in the \textsl{%
high-resolution} (or \textsl{high-frequency}) sense. This means that we
focus on the behaviour of the Fourier components%
\begin{equation}
T_{l}\left( x\right) :=\sum_{m=-l}^{l}a_{lm}Y_{lm}\left( x\right) \text{, \
\ }x\in S^{2}\text{, }l\geq 0\text{,}  \label{frequency compnts}
\end{equation}%
associated with a fixed spherical field, as the frequency $l$ grows larger
and larger (plainly, each $T_{l}$ is the projection of the field $T$ into
the orthogonal subspace of $L^{2}\left( S^{2}\right) $ spanned by the
spherical harmonics $\left\{ Y_{lm}:m=-l,...,l\right\} $). Note that this is
the typical framework faced by experimentalists handling satellite missions
as those mentioned above. Indeed, these missions are observing the same
(unique) realization of our Universe on the so-called\textsl{\ last
scattering surface}; more recent and more sophisticated experiments are then
characterized by higher and higher frequencies (smaller and smaller scales)
being observed. For instance, for the pioneering CMB mission COBE in
1989-1992 (which led to the Nobel Prize for Smoot and Mather in 2006) only
frequencies in the order of a few dozens were recorded (i.e., scales of
several degrees), a limit which was raised to few hundreds by WMAP (i.e.,
approximately a quarter of degree) and is expected to grow to a few
thousands with Planck (i.e., a few arcminutes).

The principal goal of this paper is to enlighten some partial new
connections between two high-resolution characterizations of spherical
fields, that is, \textsl{ergodicity }and \textsl{asymptotic Gaussianity}.
Roughly speaking (formal details are given in the forthcoming Sections 1.2
and 1.3), one says that the spherical field $T$ is ergodic if the empirical
version of the power spectrum of $T$ (see formula (\ref{emprical PS}) below)
can be used as a consistent estimator of the sequence $\left\{ C_{l}\right\}
$ (at least for high values of $l$). On the other hand, we say that $T$ is
asymptotically Gaussian, whenever suitably normalized versions of the
frequency components of $T_{l}$ exhibit Gaussian fluctuations for high
values of $l$. As discussed below, these two notions are tightly connected
whenever one deals with fields having an \textsl{isotropic} (or,
equivalently, \textsl{rotationally-invariant}) law.

\bigskip

\noindent%
\textbf{Remark. }For the rest of the paper, every random object is defined
on a suitable (common) probability space $\left( \Omega ,\mathcal{F}%
,P\right) .$

\subsection{High-frequency ergodicity}

In what follows, we shall consider a real-valued random field $T=\left\{
T\left( x\right) :x\in S^{2}\right\} $ indexed by the sphere $S^{2}$. The
random field $T$ satisfies the following basic assumptions: (\textbf{i}) the
law of $T$ is \textsl{isotropic}, that is, $T$ has the same law as $x\mapsto
T\left( gx\right) $ for every rotation $g\in SO\left( 3\right) $ (here, we
select the canonical action of $SO\left( 3\right) $ on $S^{2}$); (\textbf{ii}%
) $T$ is square-integrable and centered. Under assumptions (\textbf{i})-(%
\textbf{ii}), the harmonic expansion (\ref{basfor}) takes place, both in $%
L^{2}\left( P\right) $ (for fixed $x$) and in the product space $L^{2}\left(
\Omega \times S^{2},P\otimes d\lambda \right) $, where $\lambda $ stands for
the Lebesgue measure. Note that the last claim hinges on the fact that one
can regard $T$ as an application of the type $T:\Omega \times
S^{2}\rightarrow \mathbb{R}:\left( \omega ,x\right) \mapsto T\left( \omega
,x\right) $. As anticipated in the previous section, another useful property
of $T$ (easily deduced from isotropy -- see e.g. \cite{MPJMVA}) is that the
harmonic coefficients $a_{lm}$ are such that the \textsl{power spectrum }%
associated with $T$, defined as the collection $\left\{
C_{l}:l=0,1,...\right\} $ (with $C_{l}=E\left\vert a_{lm}\right\vert ^{2}$),
depends uniquely on the frequency indices $l$.

In physical experiments (for instance, when measuring the CMB\ radiation),
the power spectrum of a given spherical field is usually unknown. For this
reason, a key role is played by its empirical counterpart (called the\textsl{%
\ empirical power spectrum }-- see for instance \cite{efst},\cite{polenta}),
which is given by

\begin{equation}
\widehat{C}_{l}=\frac{1}{2l+1}\sum_{m=-l}^{l}|a_{lm}|^{2}\text{ , \ }%
l=0,1,2,...  \label{emprical PS}
\end{equation}%
An important issue to be addressed is therefore to establish conditions
under which the distance between the quantities $\widehat{C}_{l}$ and $C_{l}$
converges to zero (in a sense that is defined below) when $l\rightarrow
\infty $, that is, when higher and higher frequencies of the expansion (\ref%
{basfor}) are available to the observer$.$ Although the asymptotic behaviour
of spectrum estimators has been very deeply investigated for stochastic
processes in Euclidean domains and under large sample asymptotics (see for
instance \cite{BRD,LEO,YAD}), only basic results are known in the
high-resolution setting.

For instance, it is immediate that the finite variance of $T$ entails that,
for every $x\in S^{2}$,
\begin{equation*}
ET\left( x\right) ^{2}=\sum_{l\geq 0}\frac{(2l+1)}{4\pi }C_{l}<\infty \text{,%
}
\end{equation*}%
from which one deduces that $C_{l}\rightarrow 0$ and also
\begin{equation*}
\sum_{l\geq 0}E\widehat{C}_{l}=\sum_{l\geq 0}C_{l}<\infty \text{.}
\end{equation*}%
By reasoning as in the proof of the Borel-Cantelli Lemma we therefore infer
that, for any $\varepsilon >0$,
\begin{equation}
P\left\{ \lim_{l \rightarrow \infty }\sup \widehat{C}_{l}>\varepsilon
\right\}  \leq \lim_{l\rightarrow \infty }\sum_{\ell \geq l}P\left\{
\widehat{C}_{\ell }\geq \varepsilon \right\}  \label{yolatengo}
 \leq \lim_{l\rightarrow \infty }\frac{1}{\varepsilon }\sum_{\ell \geq
l}C_{\ell }=0,
\end{equation}
yielding in turn that both $\widehat{C}_{l}$ and $|\widehat{C}_{l}-C_{l}|$
almost surely converge to zero as $l\rightarrow \infty $. Plainly, since
this result does not provide any information about the magnitude of the
ratio $|\widehat{C}_{l}-C_{l}|/C_{l}$, it is virtually useless for
statistical applications. In particular, one cannot conclude from (\ref%
{yolatengo}) that the estimation of $C_{l}$ based on $\widehat{C}_{l}$ is
\textsl{consistent} in a satisfactory statistical sense.

Starting from these considerations, one sees that it is indeed necessary to
focus on\textsl{\ normalized} quantities, such as the sequence
\begin{equation}
\widetilde{C}_{l}=\frac{1}{2l+1}\sum_{m=-l}^{l}\frac{|a_{lm}|^{2}}{C_{l}}=%
\frac{\widehat{C}_{l}}{C_{l}},\text{ \ \ }l\geq 0\text{.}  \label{CTILDA}
\end{equation}%
Note that $E\widetilde{C}_{l}=1$, and also that the coefficient $\widetilde{C%
}_{l}$ is not observable (whereas $\widehat{C}_{l}$ is). The sequence $%
\left\{ \widetilde{C}_{l}:l\geq 0\right\} $ can be used in order to
meaningfully evaluate the asymptotic performance of any statistical
procedure based on $\widehat{C}_{l}$. The following definition uses the
coefficients $\widetilde{C}_{l}$ in order to define ergodicity.

\begin{definition}[HFE]
\label{D : HFE}Let $T$ be an isotropic, finite variance spherical random
field with angular power spectrum $\left\{ C_{l}:l\geq 0\right\} .$ We shall
say that $T$ is \textbf{High-Frequency Ergodic} (\textbf{HFE} -- or ergodic
in the high-frequency sense) if and only if%
\begin{equation}
\lim_{l\rightarrow \infty }E\left\{ \widetilde{C}_{l}-1\right\}
^{2}=\lim_{l\rightarrow \infty }E\left\{ \frac{\widehat{C}_{l}}{C_{l}}%
-1\right\} ^{2}=0\text{.}  \label{msc}
\end{equation}
\end{definition}

Condition (\ref{msc}) implies of course that $\widetilde{C}_{l}=\widehat{C}%
_{l}/C_{l}$ converges in probability towards the constant $1$.

\bigskip

\noindent \textbf{Remark. }In some sense, the term \textquotedblleft
high-frequency consistency\textquotedblright\ seems to better describe
property (\ref{msc}). However, in the statistical literature consistency is
usually viewed as a property of a sequence of estimators, whereas here we
deal with a property of the field $T$, so that we find the term ergodicity
more suitable. Another way of formulating this point is to say that (\ref%
{msc}) characterizes the ergodicity of the \textquotedblleft normalized
empirical spectral measure\textquotedblright\ $\left\{ \widetilde{C}%
_{l}:l\geq 0\right\} $, as $l$ diverges.

\subsection{Ergodicity of Gaussian fields (and associated Gaussian
fluctuations)}

As an illustration (and for future reference) we now test Definition \ref{D
: HFE} under the additional assumption that $T$ is Gaussian. In this case,
it is readily seen that, for every $l\geq 1$, the components of the vector $%
\left\{ a_{lm}:m=1,...,l\right\} $ are complex-valued and independent.
Moreover, the random quantities $a_{l0}/\sqrt{C_{l}}$, $\sqrt{2}{\rm Re}\,
(a_{lm})/\sqrt{C_{l}}$ and $\sqrt{2}{\rm Im}\,(a_{lm})/\sqrt{C_{l}}$ ($%
m=0,...,l$) are independent and identically distributed $N(0,1)$ random
variables (these facts are well-known, see e.g. \cite{BaMa},\cite{MPJMVA}
and the references therein). It is now easy to prove that%
\begin{equation}
\widetilde{C}_{l}=\frac{1}{2l+1}\sum_{m=-l}^{l}\frac{|a_{lm}|^{2}}{C_{l}}%
\rightarrow 1,  \label{LLN}
\end{equation}%
in every norm $L^{p}$, $p\geq 1$. Indeed, since $a_{l0}^{2}/C_{l}\overset{law%
}{\sim }\chi _{1}^{2}$ and the set $\left\{
2a_{lm}^{2}/C_{l}:m=1,...,l\right\} $ is composed of i.i.d. $\chi _{2}^{2}$
random variables independent of $a_{l0}$ (here, $\chi _{n}^{2}$ denotes a
standard chi-square distribution with $n$ degrees of freedom),
\begin{eqnarray*}
E\left\{ \widetilde{C}_{l}-1\right\} ^{2} &=&\frac{1}{(2l+1)^{2}}E\left[
\frac{a_{l0}^{2}}{C_{l}}-1+2\left\{ \sum_{m=1}^{l}\frac{|a_{lm}|^{2}}{C_{l}}
-1\right\} \right] ^{2} \\
&=&\frac{2}{2l+1}\underset{l\rightarrow \infty }{\longrightarrow }0,
\end{eqnarray*}
and one can use the fact that, for polynomial functionals of a Gaussian
field of fixed degree, all $L^{p}$ topologies coincide.

We shall now provide (see the forthcoming Proposition \ref{T : lodoalfano})
a CLT that is naturally associated with the convergence described in (\ref%
{LLN}). Note that, instead of using the classic Berry-Esseen results (see
e.g. Feller \cite{W. FELLER II}), we rather apply some recent estimates
(proved in \cite{Nourpec} and \cite{NourPecRein} by means of
infinite-dimensional Gaussian analysis and the so-called ``Stein's method''\
for probabilistic approximations) allowing to compare, for fixed $l$, the
\textsl{total variation distance }between the law of the normalized random
variable%
\begin{equation*}
\sqrt{\frac{2l+1}{2}}\left\{ \frac{\widehat{C}_{l}}{C_{l}}-1\right\} =\sqrt{%
\frac{2l+1}{2}}\left\{ \widetilde{C}_{l}-1\right\} ,
\end{equation*}%
and that of a standard Gaussian random variable. Recall that the total
variation distance between the laws of two real-valued random variables $X$
and $Y$ is given by
\begin{equation*}
d_{TV}\left( X,Y\right) =\sup_{A}\left| P\left( X\in A\right) -P\left( Y\in
A\right) \right| \text{,}
\end{equation*}%
where the supremum runs over all Borel sets $A$.

\begin{proposition}
\label{T : lodoalfano}Let $N\left( 0,1\right) $ denote a centered standard
Gaussian random variable. Then, for all $l\geq 0$ we have%
\begin{equation}
d_{TV}\left( \sqrt{\frac{2l+1}{2}}\left\{ \frac{\widehat{C}_{l}}{C_{l}}%
-1\right\} ,\text{ }N(0,1)\right) \leq \sqrt{\frac{8}{2l+1}}\text{,}
\label{brunetta}
\end{equation}%
so that, in particular, as $l\rightarrow \infty $,%
\begin{equation}
\sqrt{\frac{2l+1}{2}}\left\{ \widetilde{C}_{l}-1\right\} \overset{law}{%
\rightarrow }N\left( 0,1\right) .  \label{sacconi}
\end{equation}
\end{proposition}

\begin{proof}
We have%
\begin{eqnarray*}
\sqrt{\frac{2l+1}{2}}\left\{ \frac{\widehat{C}_{l}}{C_{l}}-1\right\} &=&%
\frac{1}{\sqrt{2(2l+1)}}\left\{ \frac{a_{l0}^{2}}{C_{l}}+\sum_{m=1}^{l}2%
\frac{\left\{{\rm Re}\,a_{lm}\right\} ^{2}+\left\{{\rm Im}\,a_{lm}\right\}
^{2}}{C_{l}}-(2l+1)\right\} \\
&=&\frac{1}{\sqrt{(2l+1)}}\left\{ \sum_{m=1}^{2l+1}\frac{(x_{lm}^{2}-1)}{%
\sqrt{2}}\right\} \text{ ,}
\end{eqnarray*}%
where $\left\{ x_{lm}\right\} $ are a triangular array of $i.i.d.$ standard
Gaussian random variables. Standard calculations yield that%
\begin{equation*}
cum_{4}\left\{ \frac{\sqrt{2}}{\sqrt{(2l+1)}}\left[ \sum_{m=1}^{2l+1}\frac{%
(x_{lm}^{2}-1)}{2}\right] \right\} =\frac{12}{2l+1}\text{,}
\end{equation*}%
where $cum_{j}$ stands for the $j$th cumulant. Now recall that in \cite%
{NourPecRein} it is proved that, for every zero mean and unit variance
random variable $F_{q}$ that belongs to the $q$th Wiener chaos associated
with some Gaussian field ($q\geq 2$), the following inequality holds:%
\begin{equation*}
d_{TV}(F_{q},N(0,1))\leq 2\sqrt{\frac{q-1}{3q}}\sqrt{cum_{4}\left(
F_{q}\right) }\text{ .}
\end{equation*}%
The result now follows immediately, since each variable $\sqrt{\frac{2l+1}{2}%
}\left\{ \frac{\widehat{C}_{l}}{C_{l}}-1\right\} $ has unit variance, and is
precisely an element of the second Wiener chaos associated with $T.$
\end{proof}

\bigskip

It is simple to verify numerically that the convergence (\ref{sacconi})
takes place rather fast. For instance, for $l=100$, the bound in total
variation is of the order of $2\%$, while for $l=1000$ we deduce an order of
$0.6\%$.

We stress that the previous results heavily rely on the Gaussian assumption,
and cannot be easily extended to the framework of non-Gaussian and isotropic
spherical fields. The main reason supporting this claim is contained in the
references \cite{BaMa,BaMaVa}, where it\ is shown that, under isotropy, the
coefficients $a_{lm}$ are independent if and only if the underlying field is
Gaussian, and this despite the fact that they are always uncorrelated by
construction. In other words, sampling independent, non-Gaussian random
coefficients to generate maps according to (\ref{basfor}) will always yield
an anisotropic random field. The dependence structure among the coefficients
$\left\{ a_{lm}\right\} $ is in general quite complicated, albeit it can be
neatly characterized in terms of the group representation properties of $%
SO(3)$ (see \cite{MPBER} and \cite{MPJMVA}). In view of this, to derive any
asymptotic result for $\widehat{C}_{l}$ under non-Gaussianity assumptions
for $T$, is by no means trivial and still almost completely open for
research.

\subsection{High-frequency Gaussianity}

A different form of asymptotic theory has been addressed in an apparently
unrelated stream of research, for instance in \cite{MPBER}.

\begin{definition}[HFG]
\label{D : HFG}Let $T(x)$ be an isotropic, finite variance spherical random
field, and recall the notation (\ref{basfor}) and (\ref{frequency compnts}).
We say that $T(x)$ is \textbf{high-frequency Gaussian (HFG)} whenever%
\begin{equation}
\frac{T_{l}(x)}{\sqrt{Var\left\{ T_{l}(x)\right\} }}\overset{law}{%
\rightarrow }N(0,1)\text{ , as }l\rightarrow \infty \text{, }  \label{aclt}
\end{equation}%
for every fixed $x\in S^{2}$.
\end{definition}

\bigskip

\noindent \textbf{Remark. }It is more delicate to define HFG involving
convergence in the sense of finite dimensional distributions. Indeed, in %
\cite{MPBER} it is shown that, even if relation (\ref{aclt}) holds, the
finite-dimensional distributions of order $\geq 2$ of the field $x\mapsto
T_{l}(x)/\sqrt{Var\left\{ T_{l}(x)\right\} }$ may not converge to any limit.

\bigskip

It is clear that a Gaussian field is asymptotically Gaussian: however, as
shown in \cite{MPBER}, characterizing non-Gaussian fields that are HFG can
be a difficult task, even if the underlying field $T$ is a simple
transformation (for instance, the square) of some Gaussian random function.
Conditions for the HFG property to hold in some non-Gaussian circumstances
are given in \cite{MPBER}, by using group representations -- yielding some
interesting connection with random walks on hypergroups associated with the
power spectrum of $T$. We stress that the possible existence of HFG\textit{\
}behaviour entails deep consequences on CMB data analysis. On one hand, in
fact, parameter estimation on CMB data is largely dominated by likelihood
approaches, whence an asymptotically Gaussian behaviour would great simplify
the implementation of optimal procedures. On the other hand, testing for
non-Gaussianity is a key ingredient in the validation of the so-called
\textsl{inflationary scenarios}, and the possible existence of high
frequency Gaussian components for non-Gaussian models might set a
theoretical limit to the investigation in this area.

\subsection{Purpose and plan}

Our purpose in this paper is to investigate the relationships between the
HFG and HFE properties under an assumption of Gaussian subordination, that
is, by considering fields $T$ that can be written as a deterministic
function of some isotropic, real-valued Gaussian field. We will mainly focus
on the case of polynomial subordinations, where the polynomials are of the
Hermite type. Note also that Gaussian subordination is the favoured
framework for CMB modeling in a non-Gaussian setting (see e.g. \cite{Bart},%
\cite{Hu}, \cite{yadav2}).

Our main finding is that, despite their apparent independence, the HFG and
HFE properties will turn out to be very close in a broad class of
circumstances, suggesting that ergodicity (and hence the possibility to draw
asymptotically justifiable statistical inferences) and asymptotic
Gaussianity are very tightly related in a high-resolution setting. This may
lead, we believe, to important characterizations of Gaussian random fields,
and to a better understanding of the conditions for the validity of
statistical inference procedures based on observations drawn from a unique
realizations of a compactly supported random field, as in the spherical case.

\bigskip

The plan of this paper is as follows: in Section 2 we state and prove our
main result, establishing necessary and sufficient conditions for ergodicity
and Gaussianity and exploring the link between them. Indeed, these
conditions turn out to be extremely close, so that in Section 3 we can
indeed discuss more thoroughly a special case of practical relevance, namely
the quadratic case. Section 4 is devoted to further discussion and
directions for further research.

\section{A general statement about Gaussian subordinated fields}

The two notations (\ref{basfor}) and (\ref{frequency compnts}) are adopted
throughout the sequel. Let us first recall a few basic facts and definitions.

\bigskip

\noindent (\textbf{I}) The first point concerns a characterization of
isotropy in terms of angular power spectra. Indeed, as discussed in \cite%
{Hu,MPJMVA}, if a random field is isotropic with finite fourth-order moment,
then there exists necessarily an array $\left\{ \mathcal{T}%
_{l_{1}l_{2}}^{l_{3}l_{4}}(L)\right\} $ such that%
\begin{eqnarray}
&&cum\left\{
a_{l_{1}m_{1}},a_{l_{2}m_{2}},a_{l_{3}m_{3}},a_{l_{4}m_{4}}\right\}
\label{TRISpctr} \\
&=&\sum_{LM}(-1)^{M}\left(
\begin{array}{ccc}
l_{1} & l_{2} & L \\
m_{1} & m_{2} & M%
\end{array}%
\right) \left(
\begin{array}{ccc}
l_{3} & l_{4} & L \\
m_{3} & m_{4} & -M%
\end{array}%
\right) (2L+1)\mathcal{T}_{l_{1}l_{2}}^{l_{3}l_{4}}(L)\text{ .}  \notag
\end{eqnarray}%
In general, the symbol $cum\left\{ X_{1},...,X_{m}\right\} $ denotes the
joint cumulant of the random variables $X_{1},...,X_{m}.$ Also, we label as
usual $\left\{ \mathcal{T}_{l_{1}l_{2}}^{l_{3}l_{4}}(L)\right\} $ the
cumulant trispectrum of the random field (see for instance \cite{Hu}, \cite%
{MPJMVA}); as made clear by our notation, the quantity $\mathcal{T}%
_{l_{1}l_{2}}^{l_{3}l_{4}}(L)$ does not depend on $m_{1},$ $m_{2},$ $m_{3},$
$m_{4}$ (this phenomenon is analogous to the fact that the power spectrum
only depends on the frequency $l$ -- see \cite{MPJMVA} for a discussion of
this point). On the right-hand side of (\ref{TRISpctr}), we have also
introduced the well-known Wigner's coefficients, which arise in the group
representation theory of $SO(3)$ and are discussed at length in many
excellent monographs on the quantum theory of angular momentum -- see e.g. %
\cite{LIB,VMK,VIK}. For future reference, we also recall that the Wigner's
3j coefficients are equivalent, up to normalization and phase factor, to the
Clebsch-Gordan coefficients given by (see e.g. \cite{VMK})%
\begin{equation}
C_{l_{1}m_{1}l_{2}m_{2}}^{LM}=(-1)^{l_{1}-l_{2}+m_{3}}\sqrt{2L+1}\left(
\begin{array}{ccc}
l_{1} & l_{2} & L \\
m_{1} & m_{2} & M%
\end{array}%
\right) \text{ .}  \label{noemi}
\end{equation}%
As noted by \cite{Hu}, geometrically the multipoles $%
(l_{1},l_{2},l_{3},l_{4})$ can be viewed as the sides of a quadrilateral,
and $L$ as one of its main diagonals; $L$ is also the shared size of the two
triangle formed by the corresponding pairs of side. Clebsch-Gordan
coefficients ensure that the triangle conditions are satisfied, indeed they
are different from zero only if. $l_{1}\leq l_{2}+L,$ $l_{2}\leq l_{1}+L,$
and $L\leq l_{1}+l_{2}.$

\bigskip

\noindent (\textbf{II) }We shall sometimes label a point $x$ of the sphere $%
S^{2}$ in terms of its spherical coordinates, that is, $x=\left( \vartheta
,\varphi \right) $, where $0\leq \vartheta \leq \pi $ and $0\leq \varphi
<2\pi .$

\bigskip

\noindent (\textbf{III}) Easy considerations yield the important fact that,
for any isotropic random field $T$,%
\begin{equation*}
T_{l}(\vartheta ,\varphi )\overset{law}{=}T_{l}(\overline{N}%
)=\sum_{lm}a_{lm}Y_{lm}(\overline{N})\overset{law}{=}a_{l0}\sqrt{\frac{2l+1}{%
4\pi }}\text{ ,}
\end{equation*}%
where we denote by $\overline{N}:=(0,0)$ the North Pole of the sphere and by
\textquotedblleft\ $\overset{law}{=}$ \textquotedblright\ the equality in
law between two random elements.

\bigskip

\noindent (\textbf{IV}) It is immediate that, if $T$ is isotropic, then for
every deterministic function $F$ the \textsl{subordinated} random
application $x\mapsto F\left( T\left( x\right) \right) $ is also isotropic.
Moreover, if $F(T\left( x\right) )$ is square integrable, then $F(T(\cdot ))$
also admits a harmonic expansion analogous to (\ref{basfor}). One specific
instance of this situation is obtained by choosing $T$ to be Gaussian and
isotropic, and $F$ to be any of the Hermite polynomials $\left\{ H_{q}:q\geq
0\right\} $ (in this case, one talks about a Gaussian subordination\textsl{\
}of the Hermite type). We recall that the polynomials $H_{q}$ are such that $%
H_{q}=\delta ^{q}\mathbf{1}$, where $\mathbf{1}$ stands for the function
which is constantly equal to one, $\delta ^{0}$ is the identity, and $\delta
^{q}$ ($q\geq 1$) represents the $q$th iteration of the divergence operator $%
\delta $, acting on smooth functions as $\delta f\left( x\right) =xf\left(
x\right) -f^{\prime }\left( x\right) .$ For instance $H_{0}=\mathbf{1}$, $%
H_{1}\left( x\right) =x$, $H_{2}\left( x\right) =x^{2}-1$, and so on. When $%
T $ is Gaussian, we adopt the notation%
\begin{equation}
H_{q}(T(x)):=T_{q}(x)=\sum_{l=0}^{\infty }T_{l;q}(x)\text{, \ \ }x\in S^{2}%
\text{, }q\geq 2\text{ ,}  \label{TqHAR}
\end{equation}%
where%
\begin{equation}
T_{l;q}(x)=\sum_{m=-l}^{l}a_{lm;q}Y_{lm}\left( x\right)  \label{TqFREQ}
\end{equation}%
is the $l$th frequency component of $T_{q}$, with $a_{lm;q}$ the associated
harmonic coefficients. We shall also write $\left\{ C_{l;q}:l\geq 0\right\} $
and $\left\{ \mathcal{T}_{ll}^{ll}(L;q)\right\} $, respectively, for the
power spectrum and for the cumulant trispectrum of $T_{q}$, as introduced at
Point (\textbf{I})$.$ According to \cite[Theorem 3]{MPBER}, one has that $%
C_{l;q}$ admits the following expansion in terms of the power spectrum $%
\left\{ C_{l}\right\} $ of $T$:%
\begin{equation}
C_{l;q}=q!\!\sum_{l_{1},...,l_{q}=0}^{\infty }C_{l_{1}}\!\cdot \!\cdot
\!\cdot \!C_{l_{q}}\frac{4\pi }{2l+1}\left\{ \prod_{i=1}^{q}\frac{2l_{i}+1}{%
4\pi }\right\} \sum_{L_{1}...L_{q-2}}\left\{
C_{l_{1},0;...;l_{q}0}^{L_{1},L_{2},...,L_{q-2},l;0}\right\} ^{2}\text{,}
\label{convCCC}
\end{equation}%
where $C_{l_{1},0;...;l_{q}0}^{L_{1},L_{2},...,L_{q-2},l;0}$ indicates a
convolution\textsl{\ }of Clebsch-Gordan coefficients, that is%
\begin{equation}
C_{l_{1},m_{1};...;l_{p}m_{p}}^{\lambda _{1},\lambda _{2},...,\lambda
_{p-1};\mu }\!:=\!\sum_{\mu _{1}=-\lambda _{1}}^{\lambda
_{1}}\!...\!\sum_{\mu _{p-2}=-\lambda _{p-2}}^{\lambda
_{p-2}}C_{l_{1},m_{1},l_{2},m_{2}}^{\lambda _{1},\mu _{1}}C_{\lambda
_{1},\mu _{1};l_{3},m_{3}}^{\lambda _{2},\mu _{2}}\cdot \cdot \cdot
C_{\lambda _{p-2},\mu _{p-2};l_{p},m_{p}}^{\lambda _{p-1},\mu }
\label{convo}
\end{equation}%
(see \cite{MPBER} and \cite{MPJMVA} for more details on these convolutions,
which can also be viewed as probability amplitudes in alternative coupling
schemes for quantum angular momenta, compare \cite{BIL}).

\bigskip

\noindent (\textbf{V}) An easy but important remark is the following. Since
the expansion (\ref{basfor}) is in order, the law of a centered isotropic
Gaussian field $T$ is completely encoded by the power spectrum $\left\{
C_{l}:l\geq 0\right\} $. This is a consequence of the fact that, in this
case, the array $\left\{ a_{lm}:l\geq 0,\text{ \ }m=0,...,l\right\} $ is
composed of independent Gaussian random variables such that: (i) $a_{l0}$ is
real-valued, and (ii) for every $m\geq 1$, the coefficient $a_{lm}$ has
independent and equidistributed real and imaginary parts.

\bigskip

As anticipated, we shall now prove some new connections between HFE and HFG
spherical fields (see Definitions \ref{D : HFE} and \ref{D : HFG}), in the
special case of fields of the type $T_{q}$, as defined in (\ref{TqHAR}). In
particular, our main finding (as stated in Theorem \ref{T : MAIN}) Note that
the conditions appearing in the following statement involve the coefficients
$C_{l;q}$ given in (\ref{convCCC}), and that these coefficients are
completely determined by the power spectrum of the underlying Gaussian field
$T$.

\begin{theorem}
\label{T : MAIN}Let $q\geq 2$, and define $T_{q}$ according to (\ref{TqHAR}%
), where $T$ is Gaussian and isotropic. Let $\mathcal{T}%
_{l_{1}l_{2}}^{l_{3}l_{4}}(L;q)$ be the reduced trispectrum of $T$.
Introduce the notation%
\begin{equation*}
w_{1l}(L):=\left( C_{l0l0}^{L0}\right) ^{2}\text{ \ \ and \ \ }w_{2l}(L)=%
\frac{(2L+1)}{(2l+1)^{2}}\text{,}
\end{equation*}%
in such a way that%
\begin{equation*}
\sum_{L=0}^{2l}w_{1l}(L)=\sum_{L=0}^{2l}w_{2l}(L)=1\text{ .}
\end{equation*}%
Then, the following holds.

\begin{enumerate}
\item the random field $T_{q}$ is high-frequency Gaussian if and only if
\begin{equation}
\lim_{l\rightarrow \infty }\sum_{L=0}^{2l}w_{1l}(L)\frac{\mathcal{T}%
_{ll}^{ll}(L;q)}{C_{l;q}^{2}}=0.  \label{Eq SUM1}
\end{equation}

\item On the other hand, $T_{q}$ is high-frequency ergodic if and only if%
\begin{equation}
\lim_{l\rightarrow \infty }\sum_{L=0}^{2l}w_{2l}(L)\frac{\mathcal{T}%
_{ll}^{ll}(L;q)}{C_{l;q}^{2}}=0.  \label{EqSUM2}
\end{equation}
\end{enumerate}
\end{theorem}

\bigskip

Before proving Theorem \ref{T : MAIN}, we shall note that $\left\{
C_{l0l0}^{L0}\right\} ^{2}$ is different from zero only for $L$ even, and $%
\mathcal{T}_{ll}^{ll}(L)$ is not in general positive-valued. Moreover, in
view of the forthcoming Lemma \ref{L : ODD}, also in (\ref{EqSUM2}) the sum
runs only over even values of $L.$

\begin{lemma}
\label{L : ODD}$\mathcal{T}_{ll}^{ll}(L)$ is zero when $L$ is odd
\end{lemma}

\begin{proof}
From \cite[Eq. (17)]{Hu}, we infer that, in general,%
\begin{equation*}
\mathcal{T}_{l_{1}l_{2}}^{l_{3}l_{4}}(L)=(-1)^{l_{1}+l_{2}+L}\mathcal{T}%
_{l_{1}l_{2}}^{l_{3}l_{4}}(L)\text{ .}
\end{equation*}%
Considering the case $l_{1}=l_{2}=l_{3}=l_{4}=l,$ we obtain the desired
result.
\end{proof}

\bigskip

\begin{proof}[Proof of Theorem \ref{T : MAIN}]
(\textbf{Proof of 1.}) Consider the random spherical field%
\begin{equation*}
(\vartheta ,\varphi )\mapsto \hat{T}_{l;q}(\vartheta ,\varphi ):=\frac{%
T_{l;q}(\vartheta ,\varphi )}{\sqrt{Var\left\{ T_{l;q}(\overline{N})\right\}
}}\text{, \ \ }(\vartheta ,\varphi )\in \lbrack 0,\pi ]\times \lbrack 0,2\pi
)\text{,}
\end{equation*}%
where $\overline{N}$ is the North Pole, and observe that, by isotropy and
for every $(\vartheta ,\varphi )$,
\begin{equation*}
\hat{T}_{l;q}(\vartheta ,\varphi )\overset{law}{=}\frac{a_{l0}}{\sqrt{4\pi
C_{l;q}}}\text{ }.
\end{equation*}%
The field $\hat{T}_{l;q}$ is mean-zero and has unit variance: since it also
belong to the $q$th Wiener chaos associated with $T$, we can deduce from the
results in \cite{NuPe} it is asymptotically Gaussian if and only if
\begin{equation*}
\lim_{l\rightarrow \infty }\frac{1}{C_{l;q}^{2}}cum_{4}\left\{
a_{l0;q}\right\} =0\text{ .}
\end{equation*}%
As discussed e.g. in \cite{Hu} and \cite{MPJMVA}, isotropy entails that we
can write the fourth-order cumulant as
\begin{eqnarray*}
cum_{4}\left\{ a_{l0;q}\right\} &=&\sum_{LM}(-1)^{M}\left(
\begin{array}{ccc}
l & l & L \\
0 & 0 & M%
\end{array}%
\right) \left(
\begin{array}{ccc}
l & l & L \\
0 & 0 & -M%
\end{array}%
\right) (2L+1)\mathcal{T}_{ll}^{ll}(L) \\
&=&\sum_{L}\left(
\begin{array}{ccc}
l & l & L \\
0 & 0 & 0%
\end{array}%
\right) ^{2}(2L+1)\mathcal{T}_{ll}^{ll}(L)\text{ ,}
\end{eqnarray*}%
so that the field is asymptotically Gaussian if and only if
\begin{equation}
\lim_{l\rightarrow \infty }\frac{1}{C_{l;q}^{2}}\sum_{L}\left(
\begin{array}{ccc}
l & l & L \\
0 & 0 & 0%
\end{array}%
\right) ^{2}(2L+1)\mathcal{T}_{ll}^{ll}(L)=0\text{ .}  \label{sum1}
\end{equation}%
Since relation (\ref{noemi}) is in order, we write%
\begin{equation*}
\left(
\begin{array}{ccc}
l & l & L \\
0 & 0 & 0%
\end{array}%
\right) ^{2}(2L+1)=\left\{ C_{l0l0}^{L0}\right\} ^{2}\text{,}
\end{equation*}%
entailing in turn that%
\begin{equation*}
\sum_{L}\left(
\begin{array}{ccc}
l & l & L \\
0 & 0 & 0%
\end{array}%
\right) ^{2}(2L+1)=\sum_{L=0}^{2l}\left\{ C_{l0l0}^{L0}\right\}
^{2}=\sum_{L=0}^{2l}\sum_{M=-L}^{L}\left\{ C_{l0l0}^{LM}\right\} ^{2}\equiv 1%
\text{ ,}
\end{equation*}%
where the second equality follows from the fact that Clebsch-Gordan
coefficients $C_{l_{1}m_{1}l_{2}m_{2}}^{l_{3}m_{3}}$ are different from zero
only for $m_{3}=m_{1}+m_{2}$, and the third equality is a consequence from
the orthonormality properties of the coefficients (which are the elements of
unitary matrices whose rows are indexed by $m_{1},m_{2}$ and whose columns
are indexed by $l_{3},m_{3}$)$.$ We therefore have%
\begin{equation*}
\frac{1}{C_{l;q}^{2}}cum_{4}\left\{ a_{l0;q}\right\} =\frac{1}{C_{l;q}^{2}}%
\sum_{L}\left\{ C_{l0l0}^{L0}\right\} ^{2}\mathcal{T}_{ll}^{ll}(L)\text{ ,}
\end{equation*}%
yielding the desired conclusion.

(\textbf{Proof of 2.}) On the other hand, we obtain also%
\begin{equation*}
E\left\{ \frac{\widehat{C}_{l;q}}{C_{l;q}}-1\right\} ^{2}=Var\left\{ \frac{%
\widehat{C}_{l;q}}{C_{l;q}}-1\right\}
\end{equation*}%
\begin{equation}
=\frac{1}{(2l+1)^{2}}\frac{1}{C_{l;q}^{2}}\sum_{m_{1}m_{2}}cum\left\{
a_{lm_{1};q},\overline{a}_{lm_{1};q,}a_{lm_{2};q},\overline{a}%
_{lm_{2};q}\right\} +\frac{2}{(2l+1)^{2}}\frac{1}{C_{l;q}^{2}}%
\sum_{m}\left\{ E\left\vert a_{lm;q}\right\vert ^{2}\right\} ^{2}
\label{sum0}
\end{equation}%
\begin{equation}
=\frac{1}{(2l+1)^{2}}\frac{1}{C_{l;q}^{2}}%
\sum_{m_{1}m_{2}}(-1)^{m_{1}+m_{2}}cum\left\{
a_{lm_{1};q},a_{l,-m_{1};q,}a_{lm_{2};q},a_{l,-m_{2};q}\right\} +\frac{2}{%
(2l+1)}  \label{sum00}
\end{equation}%
\begin{eqnarray}
&=&\frac{2}{(2l+1)^{2}}\frac{1}{C_{l;q}^{2}}\sum_{m_{1}m_{2}}%
\sum_{LM}(-1)^{M+m_{1}+m_{2}}\left(
\begin{array}{ccc}
l & l & L \\
m_{1} & m_{2} & M%
\end{array}%
\right) \times \notag \\
&&\times \left(
\begin{array}{ccc}
l & l & L \\
-m_{1} & -m_{2} & -M%
\end{array}%
\right) (2L+1)\mathcal{T}_{ll}^{ll}(L) +\frac{2}{(2l+1)}  \label{sum1b}
\end{eqnarray}%
\begin{equation}
=\frac{2}{(2l+1)^{2}}\frac{1}{C_{l;q}^{2}}\sum_{\substack{ L=0  \\ L\text{
even}}}^{2l}(2L+1)\mathcal{T}_{ll}^{ll}(L)+\frac{2}{(2l+1)}\text{.}
\label{sum2}
\end{equation}%
It is simple to notice that%
\begin{equation*}
\sum_{L=0}^{2l}(2L+1)=2\frac{2l(2l+1)}{2}+2l+1=(2l+1)^{2},
\end{equation*}%
so that we have%
\begin{equation*}
E\left\{ \frac{\widehat{C}_{l;q}}{C_{l;q}}-1\right\} ^{2}=2\sum_{\substack{ %
L=0  \\ L\text{ even}}}^{2l}w_{lL}^{ll}\mathcal{T}_{ll}^{ll}(L)+\frac{2}{%
(2l+1)}\text{ , where }w_{lL}\geq 0\text{ and }\sum_{\substack{ L=0  \\ L%
\text{ even}}}^{2l}w_{lL}=1\text{.}
\end{equation*}%
The result now follows immediately.
\end{proof}

\bigskip

\noindent \textbf{Remark. }Note that%
\begin{equation*}
\left\{ C_{l0l0}^{L0}\right\} ^{2}=\frac{(2L+1)(\frac{2l+L}{2}!)^{2}}{(\frac{%
L}{2}!)^{2}}\frac{(L!)^{2}(2l-L)!}{(2l+L+1)!}\leq \frac{1}{(2L+1)}
\end{equation*}%
\begin{equation*}
w_{2l}(L)=\frac{(2L+1)}{(2l+1)^{2}}\leq \frac{1}{2l+1}
\end{equation*}%
Note also that in the Gaussian case (e.g., $q=1$) we have $\mathcal{T}%
_{ll}^{ll}(L)\equiv 0,$ whence
\begin{equation*}
E\left\{ \frac{\widehat{C}_{l;q}}{C_{l;q}}-1\right\} ^{2}=\frac{2}{(2l+1)}%
\rightarrow 0\text{ ,}
\end{equation*}%
as expected.

\bigskip

The previous result strongly suggests that the conditions for asymptotic
Gaussianity (HFG) and for ergodicity (HFE) should be tightly related. Indeed
we conjecture that HFE and HFG are equivalent in the case of Hermite type
Gaussian subordinations (and most probably even in more general
circumstances). However, proving this claim seems analytically too demanding
at this stage, so that for the rest of the paper we content ourselves with a
detailed analysis of \textsl{quadratic Gaussian subordinations. }In
particular, we believe that the content of the forthcoming Section \ref%
{noemi} (which is already quite technical) may provide the seed for a
complete understanding of the HFG-HFE connection.

\bigskip

\begin{remark}
It should be noted that the reduced trispectrum satisfies (see \cite[Eq.
(16)]{Hu})
\begin{equation*}
\mathcal{T}_{ll}^{ll}(L^{\prime })=\sum_{L}(2L+1)\left\{
\begin{tabular}{ccc}
$l$ & $l$ & $L$ \\
$l$ & $l$ & $L^{\prime }$%
\end{tabular}%
\right\} \mathcal{T}_{ll}^{ll}(L)\text{ .}
\end{equation*}
\end{remark}

In the previous remark, we introduced the well-known Wigner's 6j
coefficients, which intertwine alternative coupling schemes of three quantum
angular momenta (see \cite{BIL}, \cite{VMK} for futher properties and much
more discussion). Their relationship with Wigner's 3j coefficients is
provided by the identity
\begin{equation}
\left\{
\begin{array}{ccc}
a & b & e \\
c & d & f%
\end{array}%
\right\} :=\sum_{\substack{ \alpha ,\beta ,\gamma  \\ \varepsilon ,\delta
,\phi }}(-1)^{e+f+\varepsilon +\phi }\left(
\begin{array}{ccc}
a & b & e \\
\alpha & \beta & \varepsilon%
\end{array}%
\right) \left(
\begin{array}{ccc}
c & d & e \\
\gamma & \delta & -\varepsilon%
\end{array}%
\right) \left(
\begin{array}{ccc}
a & d & f \\
\alpha & \delta & -\phi%
\end{array}%
\right) \left(
\begin{array}{ccc}
c & b & f \\
\gamma & \beta & \phi%
\end{array}%
\right)  \label{6j1}
\end{equation}%
(see \cite{VMK}, Chapter 9, for analytic expressions and a full set of
properties).

\section{The quadratic case\label{S : quadratics}}

\subsection{The class $\mathfrak{D}$ and main results}

As anticipated, the purpose of this section is to provide a more detailed
and explicit analysis of the quadratic case $q=2.$ For simplicity, in the
sequel we consider a centered Gaussian isotropic spherical field $T$ such
that $Var(T(x))=\sum_{l}(2l+1)C_{l}/4\pi =1$, where $\left\{ C_{l}\right\} $
is as before the power spectrum of $T.$ We start by recalling the notation
\begin{equation}
T_{2}(x)=H_{2}(T(x))=\sum_{l_{1},l_{2}=1}^{\infty
}\sum_{m_{1}m_{2}}a_{l_{1}m_{1}}a_{l_{2}m_{2}}Y_{l_{1}m_{1}}(x)Y_{l_{2}m_{2}}(x)-1%
\text{ ,}  \label{T-H2}
\end{equation}%
where $T$ is isotropic, centered and Gaussian. Our first result can be seen
as a consequence of formula (\ref{convCCC}) (or, more generally, of the
results of \cite{MPBER}). Here, we provide a proof for the sake of
completeness.

\begin{lemma}
The angular power spectrum of the squared random field (\ref{T-H2}) is given
by%
\begin{equation*}
C_{l;2}=E%
{\vert}%
a_{lm;2}%
{\vert}%
^{2}=2\sum_{l_{1}l_{2}}C_{l_{1}}C_{l_{2}}\left(
\begin{array}{ccc}
l_{1} & l_{2} & l \\
0 & 0 & 0%
\end{array}%
\right) ^{2}\frac{(2l+1)(2l+1)}{4\pi }\text{ .}
\end{equation*}
\end{lemma}

\begin{proof}
Recall first that $Y_{00}(x)\equiv (4\pi )^{-1/2},$ see \cite{VMK}, equation
5.13.1.1. Hence, in view of (\ref{T-H2}), we have that, for $l=0$,%
\begin{eqnarray*}
a_{00;2} &=&\int_{S^{2}}\left\{
\sum_{l_{1}l_{2}}%
\sum_{m_{1}m_{2}}a_{l_{1}m_{1}}a_{l_{2}m_{2}}Y_{l_{1}m_{1}}(x)Y_{l_{2}m_{2}}(x)-1\right\}
\overline{Y}_{00}(x)dx \\
&=&\frac{1}{\sqrt{4\pi }}\sum_{l_{1}l_{2}}\sum_{m_{1}m_{2}}a_{l_{1}m_{1}}%
\overline{a}_{l_{2}m_{2}}\left\{ \int_{S^{2}}Y_{l_{1}m_{1}}(x)\overline{Y}%
_{l_{2}m_{2}}(x)dx\right\} -\sqrt{4\pi } \\
&=&\frac{1}{\sqrt{4\pi }}\sum_{l_{1}l_{2}}\sum_{m_{1}m_{2}}a_{l_{1}m_{1}}%
\overline{a}_{l_{2}m_{2}}\delta _{l_{1}}^{l_{2}}\delta _{m_{1}}^{m_{2}}-%
\sqrt{4\pi } \\
&=&\frac{1}{\sqrt{4\pi }}\sum_{lm}%
{\vert}%
a_{lm}%
{\vert}%
^{2}-\sqrt{4\pi }\text{ .}
\end{eqnarray*}%
It follows that
\begin{equation*}
Ea_{00;2}=\sum_{lm}\frac{2l+1}{\sqrt{4\pi }}C_{l}-\sqrt{4\pi }=\sqrt{4\pi }%
\left\{ \sum_{lm}\frac{2l+1}{4\pi }C_{l}-1\right\} =0\text{ ,}
\end{equation*}%
and
\begin{equation*}
EH_{2}(T(x))=E\sum_{l=0}^{\infty }a_{lm;2}Y_{lm}(x)=Ea_{00;2}Y_{00}(x)=0%
\text{ ,}
\end{equation*}%
the second step following because $Ea_{lm}=0$ for all $l>0$ under isotropy
(see \cite{BaMa}). Indeed we have (from (\ref{T-H2}), and in view of (\ref%
{TqFREQ}))%
\begin{equation*}
a_{lm;2}=\int_{S^{2}}\sum_{l_{1}l_{2}}%
\sum_{m_{1}m_{2}}a_{l_{1}m_{1}}a_{l_{2}m_{2}}Y_{lm_{1}}(x)Y_{lm_{2}}(x)%
\overline{Y}_{lm}(x)dx
\end{equation*}%
\begin{equation}
=\sum_{l_{1},l_{2}=1}^{\infty
}\sum_{m_{1}m_{2}}a_{l_{1}m_{1}}a_{l_{2}m_{2}}\left(
\begin{array}{ccc}
l_{1} & l_{2} & l \\
m_{1} & m_{2} & -m%
\end{array}%
\right) \left(
\begin{array}{ccc}
l_{1} & l_{2} & l \\
0 & 0 & 0%
\end{array}%
\right) \sqrt{\frac{(2l_{1}+1)(2l_{2}+1)(2l+1)}{4\pi }}\text{ .}
\label{Gauntino}
\end{equation}%
Note that the constant term $-1$ has no effect \ for $l\geq 1$, because
\begin{equation*}
\int_{S^{2}}Y_{lm}(x)dx=0\text{ for all }l\geq 1\text{ .}
\end{equation*}%
Now
\begin{eqnarray*}
Ea_{lm;2} &=&\sum_{l_{1}}\sum_{m_{1}}C_{l_{1}}(-1)^{m_{1}}\left(
\begin{array}{ccc}
l_{1} & l_{2} & l \\
m_{1} & -m_{1} & -m%
\end{array}%
\right) \left(
\begin{array}{ccc}
l_{1} & l_{2} & l \\
0 & 0 & 0%
\end{array}%
\right) \sqrt{\frac{(2l_{1}+1)(2l_{2}+1)(2l+1)}{4\pi }} \\
&=&\sum_{l_{1}}C_{l_{1}}\left(
\begin{array}{ccc}
l_{1} & l_{1} & l \\
0 & 0 & 0%
\end{array}%
\right) \sqrt{\frac{(2l_{1}+1)^{2}(2l+1)}{4\pi }}\delta
_{m}^{0}\sum_{m_{1}}(-1)^{m_{1}}\left(
\begin{array}{ccc}
l_{1} & l_{1} & \ell \\
m_{1} & -m_{1} & 0%
\end{array}%
\right) \\
&=&\sum_{l_{1}}C_{l_{1}}\left(
\begin{array}{ccc}
l_{1} & l_{1} & l \\
0 & 0 & 0%
\end{array}%
\right) \sqrt{\frac{(2l_{1}+1)^{2}(2l+1)}{4\pi }}\sqrt{2l_{1}+1}\delta
_{m}^{0}\delta _{l}^{0}\text{ ,}
\end{eqnarray*}%
in view of the well-known properties (\cite[Eq. (8.5.1.1) and Eq. (8.7.1.2)]%
{VMK})%
\begin{equation*}
\left(
\begin{array}{ccc}
l_{1} & l_{1} & 0 \\
0 & 0 & 0%
\end{array}%
\right) =\frac{1}{\sqrt{2l_{1}+1}}\text{ , }\sum_{m_{1}}(-1)^{m_{1}}\left(
\begin{array}{ccc}
l_{1} & l_{1} & l \\
m_{1} & -m_{1} & 0%
\end{array}%
\right) =\sqrt{2l_{1}+1}\delta _{l}^{0}\text{ .}
\end{equation*}%
Hence we have, as expected, $Ea_{lm;2}=0$ for all $l.$ Furthermore%
\begin{eqnarray*}
E\left| a_{lm;2}\right| ^{2} &=&E\left\{
\sum_{l_{1}l_{2}}\sum_{m_{1}m_{2}}a_{l_{1}m_{1}}a_{l_{2}m_{2}}\left(
\begin{array}{ccc}
l_{1} & l_{2} & l \\
m_{1} & m_{2} & -m%
\end{array}%
\right) \left(
\begin{array}{ccc}
l_{1} & l_{2} & l \\
0 & 0 & 0%
\end{array}%
\right) \sqrt{\frac{(2l_{1}+1)(2l_{2}+1)(2l+1)}{4\pi }}\right. \\
&&\left. \sum_{l_{1}^{\prime }l_{2}^{\prime }}\sum_{m_{1}^{\prime
}m_{2}^{\prime }}\overline{a}_{l_{1}^{\prime }m_{1}^{\prime }}\overline{a}%
_{l_{2}^{\prime }m_{2}^{\prime }}\left(
\begin{array}{ccc}
l_{1}^{\prime } & l_{2}^{\prime } & l \\
m_{1}^{\prime } & m_{2}^{\prime } & -m%
\end{array}%
\right) \left(
\begin{array}{ccc}
l_{1}^{\prime } & l_{2}^{\prime } & l \\
0 & 0 & 0%
\end{array}%
\right) \sqrt{\frac{(2l_{1}^{\prime }+1)(2l_{2}^{\prime }+1)(2l+1)}{4\pi }}%
\right\}
\end{eqnarray*}%
\begin{equation*}
=2\sum_{l_{1}l_{2}}C_{l_{1}}C_{l_{2}}\sum_{m_{1}m_{2}}\left(
\begin{array}{ccc}
l_{1} & l_{2} & l \\
m_{1} & m_{2} & -m%
\end{array}%
\right) ^{2}\left(
\begin{array}{ccc}
l_{1} & l_{2} & l \\
0 & 0 & 0%
\end{array}%
\right) ^{2}\frac{(2l_{1}+1)(2l_{2}+1)(2l+1)}{4\pi }
\end{equation*}%
\begin{equation*}
=2\sum_{l_{1}l_{2}}C_{l_{1}}C_{l_{2}}\left(
\begin{array}{ccc}
l_{1} & l_{2} & l \\
0 & 0 & 0%
\end{array}%
\right) ^{2}\frac{(2l_{1}+1)(2l_{2}+1)}{4\pi }\text{ ,}
\end{equation*}%
and the proof is completed.
\end{proof}

\bigskip

\noindent \textbf{Remark. }Note that
\begin{equation*}
Var\left\{ T^{2}(x)\right\} =\sum_{l}\frac{2l+1}{4\pi }C_{l}
\end{equation*}%
\begin{eqnarray*}
&=&2\sum_{l_{1}l_{2}}C_{l_{1}}C_{l_{2}}\frac{(2l_{1}+1)(2l_{2}+1)}{4\pi }%
\left\{ \sum_{l}\frac{2l+1}{4\pi }\left(
\begin{array}{ccc}
l_{1} & l_{2} & l \\
0 & 0 & 0%
\end{array}%
\right) ^{2}\right\} \\
&=&2\sum_{l_{1}l_{2}}C_{l_{1}}C_{l_{2}}\frac{(2l_{1}+1)(2l_{2}+1)}{(4\pi
)^{2}}=2\left[ Var\left\{ T(x)\right\} \right] ^{2}\text{ ,}
\end{eqnarray*}%
as expected from standard property of Gaussian variables. Here we have used
again%
\begin{equation*}
\sum_{l}(2l+1)\left(
\begin{array}{ccc}
l_{1} & l_{2} & l \\
0 & 0 & 0%
\end{array}%
\right) ^{2}\equiv 1\text{ .}
\end{equation*}

\bigskip

Our strategy is now the following. We shall first define a very general
class, noted $\mathfrak{D}$, of quadratic models in terms of the power
spectrum of the underlying Gaussian field, and then we shall show that the
two notions of HFG and HFE coincide within $\mathfrak{D}$.

\bigskip

\begin{definition}
\label{D : la classe D}The centered Gaussian isotropic field $T$ is said to
belong to the class $\mathfrak{D}$ if there exist real numbers $\alpha
,\beta $ such that

\begin{enumerate}
\item $\alpha \in \mathbb{R}$ and $\beta \geq 0$

\item $\sum_{l=0}^{\infty }l^{-\alpha +1}e^{-\beta l}<\infty $

\item there exists constants $c_{1},c_{2}>0$ such that%
\begin{equation}
0<c_{1}\leq \lim \inf_{l\rightarrow \infty }\frac{C_{l}}{l^{-\alpha
}e^{-\beta l}}\leq \lim \sup_{l\rightarrow \infty }\frac{C_{l}}{l^{-\alpha
}e^{-\beta l}}\leq c_{2}<\infty   \label{Eq classe D}
\end{equation}
\end{enumerate}
\end{definition}

\bigskip

\noindent \textbf{Remarks. }(1) As a first approximation, the class $%
\mathfrak{D}$ contains virtually all models that are relevant for CMB\
modeling in the case of a quadratic Gaussian subordination. For instance,
Sachs-Wolfe models with the so-called Bardeen's potential entail a
polynomial decay of the $C_{l}$ ($\beta =0$), whereas the so-called \emph{%
Silk damping} effect entails an exponential decay of the power spectrum of
primary CMB anisotropies at higher $l.$ We refer again to textbooks such as %
\cite{DOD}, \cite{Durren} for more discussion on these points.

(2) Note that Condition 2 in the definition of $\mathfrak{D}$ implies that
the parameters $\alpha ,\beta $ must be such that either $\beta =0$ and $%
\alpha >2$, or $\beta >0$ and $\alpha \in \mathbb{R}$ (with no restrictions).

\bigskip

The next statement is the main achievement of this section. It shows in
particular, that the HFG and HFE exhibit the same phase transition within
the class $\mathfrak{D}$

\begin{theorem}
\label{T : MAINquad}Let $T_{2}=H_{2}\left( T\right) $, where the centered
Gaussian isotropic field $T$ is an element of the class $\mathfrak{D}$.
Then, the following three conditions are equivalent

\begin{description}
\item[\textrm{(i)}] $T_{2}$ is HFG

\item[\textrm{(ii)}] $T_{2}$ is HFE

\item[\textrm{(iii)}] $\beta >0$ and $\alpha \in \mathbb{R}$.
\end{description}
\end{theorem}

\bigskip

\subsection{Proof of Theorem \ref{T : MAINquad}}

From \cite[Section 6]{MPBER}, we already know that Conditions (i) and (iii)
in the statement of Theorem \ref{T : MAINquad} are equivalent. The proof of
the remaining implication (ii) $\Longleftrightarrow $ (iii) is divided in
several steps.

We start by showing that, if (iii) is not verified, then the angular power
spectrum of the transformed field, under broad conditions, exhibits the same
behaviour as the angular power spectrum of the subordinating field.

\begin{lemma}
\label{L : zeta}Suppose $\beta =0$ and $\alpha >2$, then%
\begin{equation*}
\frac{3\times 2^{\alpha }}{4\pi }C_{l}\frac{c_{1}^{2}}{c_{2}}\leq
C_{l;2}\leq \frac{c_{2}^{2}}{c_{1}\pi }\left\{ 2\zeta (\alpha -1)+\zeta
(\alpha )\right\} C_{l/2}=O(C_{l})\text{ ,}
\end{equation*}%
where $\zeta (.)$ denotes the Riemann zeta function.
\end{lemma}

\begin{proof}
We have%
\begin{equation*}
\sum_{l_{1}l_{2}}C_{l_{1}}C_{l_{2}}\left(
\begin{array}{ccc}
l_{1} & l_{2} & l \\
0 & 0 & 0%
\end{array}%
\right) ^{2}\frac{(2l_{1}+1)(2l_{2}+1)}{4\pi }\leq 2\sum_{l_{1}\leq
l_{2}}C_{l_{1}}C_{l_{2}}\left(
\begin{array}{ccc}
l_{1} & l_{2} & l \\
0 & 0 & 0%
\end{array}%
\right) ^{2}\frac{(2l_{1}+1)(2l_{2}+1)}{4\pi }
\end{equation*}%
\begin{equation*}
\leq 2\frac{c_{2}}{c_{1}}C_{l/2}\sum_{l_{1}\leq l_{2}}C_{l_{1}}\left(
\begin{array}{ccc}
l_{1} & l_{2} & l \\
0 & 0 & 0%
\end{array}%
\right) ^{2}\frac{(2l_{1}+1)(2l_{2}+1)}{4\pi }\text{ ,}
\end{equation*}%
because $(l_{1}\vee l_{2})>l/2$ by the triangle conditions and $%
\sup_{l_{2}\geq l/2}C_{l_{2}}/C_{l/2}\leq c_{2}/c_{1}$. Now%
\begin{equation*}
\sum_{l_{1}\leq l_{2}}C_{l_{1}}\left(
\begin{array}{ccc}
l_{1} & l_{2} & l \\
0 & 0 & 0%
\end{array}%
\right) ^{2}\frac{(2l_{1}+1)(2l_{2}+1)}{4\pi }
\end{equation*}%
\begin{equation*}
\leq \sum_{l_{1}}C_{l_{1}}\frac{(2l_{1}+1)}{4\pi }\sum_{l_{2}}(2l_{2}+1)%
\left(
\begin{array}{ccc}
l_{1} & l_{2} & l \\
0 & 0 & 0%
\end{array}%
\right) ^{2}=\sum_{l_{1}}C_{l_{1}}\frac{(2l_{1}+1)}{4\pi }<\infty \text{ .}
\end{equation*}%
More precisely%
\begin{equation*}
\sum_{l_{1}}C_{l_{1}}\frac{(2l_{1}+1)}{4\pi }\leq \frac{c_{2}}{4\pi }%
\sum_{l}(2l+1)l^{-\alpha }\leq \frac{c_{2}}{4\pi }\left\{ 2\zeta (\alpha
-1)+\zeta (\alpha )\right\} \text{ .}
\end{equation*}%
Hence%
\begin{equation*}
C_{l;2}\leq \frac{c_{2}^{2}}{2c_{1}\pi }\left\{ 2\zeta (\alpha -1)+\zeta
(\alpha )\right\} C_{l/2}\text{ .}
\end{equation*}%
The upper bound is then established. For the lower bound, it is sufficient
to show that%
\begin{equation*}
\sum_{l_{1}l_{2}}C_{l_{1}}C_{l_{2}}\left(
\begin{array}{ccc}
l_{1} & l_{2} & l \\
0 & 0 & 0%
\end{array}%
\right) ^{2}\frac{(2l_{1}+1)(2l_{2}+1)}{4\pi }\geq
\sum_{l_{2}}C_{1}C_{l_{2}}\left(
\begin{array}{ccc}
l_{1} & l_{2} & l \\
0 & 0 & 0%
\end{array}%
\right) ^{2}\frac{3(2l_{2}+1)}{4\pi }
\end{equation*}%
\begin{equation*}
\geq 32^{\alpha }C_{l}\frac{c_{1}^{2}}{c_{2}}\sum_{l_{2}}\left(
\begin{array}{ccc}
l_{1} & l_{2} & l \\
0 & 0 & 0%
\end{array}%
\right) ^{2}\frac{(2l_{2}+1)}{4\pi }\geq \frac{3\times 2^{\alpha }}{4\pi }%
C_{l}\frac{c_{1}^{2}}{c_{2}}\text{ ,}
\end{equation*}%
as claimed.
\end{proof}

\bigskip

Loosely, the previous Lemma \ref{L : zeta} states that, under algebraic
decay, the rate of convergence to zero of the angular power spectrum is not
affected by a quadratic transformation, i.e. $C_{l;2}\simeq C_{l}$. The
following result holds for fixed $l$, and it is therefore not related to the
high-frequency asymptotic behaviour of the power spectrum $\left\{
C_{l}\right\} $ (see \cite{MPBER} for related computations). Note that we
use the notation
\begin{equation*}
\widehat{C}_{l;2}=\frac{1}{2l+1}\sum_{m=-l}^{l}|a_{lm;2}|^{2},\text{ }%
\widetilde{C}_{l;2}=\frac{\widehat{C}_{l;2}}{C_{l;2}}\text{ }.
\end{equation*}

\begin{lemma}
\label{sabato} Let $T_{2}$ be defined by (\ref{T-H2}). Then we have%
\begin{equation*}
E\left\{ \widetilde{C}_{l;2}-1\right\} ^{2}
\end{equation*}%
\begin{equation*}
=\frac{16}{C_{l;2}^{2}}\sum_{l_{1}l_{2}l_{3}}C_{l_{1}}^{2}C_{l_{2}}C_{l_{3}}%
\left(
\begin{array}{ccc}
l_{1} & l_{2} & l \\
0 & 0 & 0%
\end{array}%
\right) ^{2}\left(
\begin{array}{ccc}
l_{1} & l_{3} & l \\
0 & 0 & 0%
\end{array}%
\right) ^{2}\frac{(2l_{1}+1)(2l_{2}+1)(2l_{3}+1)}{(4\pi )^{2}}+R(l)\text{ ,}
\end{equation*}%
where for all $l=1,2,....$%
\begin{equation*}
0\leq R(l)\leq \frac{4}{2l+1}\text{ .}
\end{equation*}
\end{lemma}

\begin{proof}
In the sequel, we shall use repeatedly the unitary properties of
Clebsch-Gordan coefficients, i.e.
\begin{equation}
\sum_{m_{1}m_{2}}\left(
\begin{array}{ccc}
l & l & L \\
m_{1} & m_{2} & M%
\end{array}%
\right) \left(
\begin{array}{ccc}
l & l & L^{\prime } \\
m_{1} & m_{2} & M^{\prime }%
\end{array}%
\right) =\frac{\delta _{L}^{L^{\prime }}\delta _{M}^{M^{\prime }}}{2L+1}%
\text{ .}  \label{unicg}
\end{equation}%
Recalling \ref{sum0}, \ref{sum00}, we need to evaluate
\begin{eqnarray*}
&&\frac{1}{(2l+1)^{2}C_{l;2}^{2}}\sum_{m_{1}m_{2}}cum\left\{ a_{lm_{1}},%
\overline{a}_{lm_{1}},a_{lm_{2},}\overline{a}_{lm_{2}}\right\} \\
&=&\frac{1}{(2l+1)^{2}C_{l;2}^{2}}\sum_{m_{1}m_{2}}(-1)^{m_{1}+m_{2}}cum%
\left\{ a_{lm_{1}},a_{l,-m_{1}},a_{lm_{2},}a_{l,-m_{2}}\right\} \text{ .}
\end{eqnarray*}%
Now%
\begin{equation*}
cum\left\{ a_{lm_{1}},a_{l,-m_{1}},a_{lm_{2},}a_{l,-m_{2}}\right\} =
\end{equation*}%
\begin{eqnarray*}
&=&cum\left\{ \sum_{l_{1}l_{2}}\sum_{\mu _{1}\mu _{2}}a_{l_{1}\mu
_{1}}a_{l_{2}\mu _{2}}\left(
\begin{array}{ccc}
l_{1} & l_{2} & l \\
\mu _{1} & \mu _{2} & m_{1}%
\end{array}%
\right) \left(
\begin{array}{ccc}
l_{1} & l_{2} & l \\
0 & 0 & 0%
\end{array}%
\right) \sqrt{\frac{(2l_{1}+1)(2l_{2}+1)(2l+1)}{4\pi }},\right. \\
&&\sum_{l_{3}l_{4}}\sum_{\mu _{3}\mu _{4}}a_{l_{3}\mu _{3}}a_{l_{4}\mu
_{4}}\left(
\begin{array}{ccc}
l_{3} & l_{4} & l \\
\mu _{3} & \mu _{4} & -m_{1}%
\end{array}%
\right) \left(
\begin{array}{ccc}
l_{3} & l_{4} & l \\
0 & 0 & 0%
\end{array}%
\right) \sqrt{\frac{(2l_{3}+1)(2l_{4}+1)(2l+1)}{4\pi }},
\end{eqnarray*}%
\begin{eqnarray*}
&&\sum_{l_{5}l_{6}}\sum_{\mu _{5}\mu _{6}}a_{l\mu _{5}}a_{l\mu _{6}}\left(
\begin{array}{ccc}
l_{5} & l_{6} & l \\
\mu _{5} & \mu _{6} & m_{2}%
\end{array}%
\right) \left(
\begin{array}{ccc}
l_{5} & l_{6} & l \\
0 & 0 & 0%
\end{array}%
\right) \sqrt{\frac{(2l_{5}+1)(2l_{6}+1)(2l+1)}{4\pi }}, \\
&&\left. \sum_{l_{7}l_{8}}\sum_{\mu _{7}\mu _{8}}a_{l\mu _{7}}a_{l\mu
_{8}}\left(
\begin{array}{ccc}
l_{7} & l_{8} & l \\
\mu _{7} & \mu _{8} & -m_{2}%
\end{array}%
\right) \left(
\begin{array}{ccc}
l_{7} & l_{8} & l \\
0 & 0 & 0%
\end{array}%
\right) \sqrt{\frac{(2l_{7}+1)(2l_{8}+1)(2l+1)}{4\pi }}\right\}
\end{eqnarray*}%
and counting equivalent permutations%
\begin{equation*}
=8\sum_{l_{1}l_{2}l_{3}l_{4}}\sum_{\mu _{1}\mu _{2}\mu _{3}\mu
_{4}}(-1)^{\mu _{1}+\mu _{2}+\mu _{3}+\mu
_{4}}C_{l_{1}}C_{l_{2}}C_{l_{3}}C_{l_{4}}\left(
\begin{array}{ccc}
l_{1} & l_{2} & l \\
\mu _{1} & \mu _{2} & m_{1}%
\end{array}%
\right) \left(
\begin{array}{ccc}
l_{1} & l_{2} & l \\
0 & 0 & 0%
\end{array}%
\right) \frac{(2l+1)^{2}\prod\limits_{i=1}^{4}(2l_{i}+1)}{(4\pi )^{2}}
\end{equation*}%
\begin{equation*}
\times \left(
\begin{array}{ccc}
l_{1} & l_{3} & l \\
-\mu _{1} & -\mu _{3} & -m_{1}%
\end{array}%
\right) \left(
\begin{array}{ccc}
l_{1} & l_{3} & l \\
0 & 0 & 0%
\end{array}%
\right) \left(
\begin{array}{ccc}
l_{4} & l_{3} & l \\
\mu _{4} & \mu _{3} & m_{2}%
\end{array}%
\right) \left(
\begin{array}{ccc}
l_{4} & l_{3} & l \\
0 & 0 & 0%
\end{array}%
\right) \left(
\begin{array}{ccc}
l_{4} & l_{2} & l \\
-\mu _{4} & -\mu _{2} & m_{2}%
\end{array}%
\right) \left(
\begin{array}{ccc}
l_{4} & l_{2} & l \\
0 & 0 & 0%
\end{array}%
\right)
\end{equation*}%
\begin{equation*}
+8\sum_{l_{1}l_{2}l_{3}l_{4}}\sum_{\mu _{1}\mu _{2}\mu _{3}\mu
_{4}}(-1)^{\mu _{1}+\mu _{2}+\mu _{3}+\mu
_{4}}C_{l_{1}}C_{l_{2}}C_{l_{3}}C_{l_{4}}\left(
\begin{array}{ccc}
l_{1} & l_{2} & l \\
\mu _{1} & \mu _{2} & m_{1}%
\end{array}%
\right) \left(
\begin{array}{ccc}
l_{1} & l_{2} & l \\
0 & 0 & 0%
\end{array}%
\right) \frac{(2l+1)^{2}\prod\limits_{i=1}^{4}(2l_{i}+1)}{(4\pi )^{2}}
\end{equation*}%
\begin{equation*}
\times \left(
\begin{array}{ccc}
l_{3} & l_{4} & l \\
\mu _{3} & \mu _{4} & -m_{1}%
\end{array}%
\right) \left(
\begin{array}{ccc}
l_{3} & l_{4} & l \\
0 & 0 & 0%
\end{array}%
\right) \left(
\begin{array}{ccc}
l_{1} & l_{3} & l \\
-\mu _{1} & -\mu _{3} & m_{2}%
\end{array}%
\right) \left(
\begin{array}{ccc}
l_{1} & l_{3} & l \\
0 & 0 & 0%
\end{array}%
\right) \left(
\begin{array}{ccc}
l_{4} & l_{2} & l \\
-\mu _{4} & -\mu _{2} & -m_{2}%
\end{array}%
\right) \left(
\begin{array}{ccc}
l_{4} & l_{2} & l \\
0 & 0 & 0%
\end{array}%
\right)
\end{equation*}%
\begin{equation*}
+8\sum_{l_{1}l_{2}l_{3}l_{4}}\sum_{\mu _{1}\mu _{2}\mu _{3}\mu
_{4}}(-1)^{\mu _{1}+\mu _{2}+\mu _{3}+\mu
_{4}}C_{l_{1}}C_{l_{2}}C_{l_{3}}C_{l_{4}}\left(
\begin{array}{ccc}
l_{1} & l_{2} & l \\
\mu _{1} & \mu _{2} & m_{1}%
\end{array}%
\right) \times
\end{equation*}%
\begin{equation*}
\times \left(
\begin{array}{ccc}
l_{1} & l_{2} & l \\
0 & 0 & 0%
\end{array}%
\right) \left(
\begin{array}{ccc}
l_{1} & l_{3} & l \\
-\mu _{1} & \mu _{3} & -m_{1}%
\end{array}%
\right) \left(
\begin{array}{ccc}
l_{1} & l_{3} & l \\
0 & 0 & 0%
\end{array}%
\right)\times
\end{equation*}
\begin{equation*}
\times \left(
\begin{array}{ccc}
l_{2} & l_{4} & l \\
-\mu _{2} & \mu _{4} & m_{2}%
\end{array}%
\right) \left(
\begin{array}{ccc}
l_{2} & l_{4} & l \\
0 & 0 & 0%
\end{array}%
\right) \left(
\begin{array}{ccc}
l_{3} & l_{4} & l \\
-\mu _{3} & -\mu _{4} & -m_{2}%
\end{array}%
\right) \left(
\begin{array}{ccc}
l_{4} & l_{2} & l \\
0 & 0 & 0%
\end{array}%
\right) \frac{(2l+1)^{2}\prod\limits_{i=1}^{4}(2l_{i}+1)}{(4\pi )^{2}}
\end{equation*}%
\begin{equation*}
=:8\left\{
A(m_{1},-m_{1},m_{2},-m_{2})+B(m_{1},-m_{1},m_{2},-m_{2})+C(m_{1},-m_{1},m_{2},-m_{2})\right\}
\text{ .}
\end{equation*}%
For the first term, note first that $(-1)^{m_{1}+m_{2}+\mu _{1}+\mu _{2}+\mu
_{3}+\mu _{4}}\equiv 1,$ because the exponent is necessarily even by the
properties of Wigner's coefficients. Moreover, applying iteratively (\ref%
{unicg})%
\begin{equation*}
\sum_{m_{1}m_{2}}A(m_{1},-m_{1},m_{2},-m_{2})
\end{equation*}%
\begin{equation*}
=\sum_{l_{1}l_{2}l_{3}l_{4}}\sum_{m_{2}\mu _{2}}\sum_{\mu _{3}\mu
_{4}}C_{l_{1}}C_{l_{2}}C_{l_{3}}C_{l_{4}}\left(
\begin{array}{ccc}
l_{1} & l_{2} & l \\
0 & 0 & 0%
\end{array}%
\right) \left(
\begin{array}{ccc}
l_{1} & l_{3} & l \\
0 & 0 & 0%
\end{array}%
\right) \frac{(2l+1)^{2}\prod\limits_{i=1}^{4}(2l_{i}+1)}{(4\pi )^{2}}
\end{equation*}%
\begin{equation*}
\times \left(
\begin{array}{ccc}
l_{4} & l_{3} & l \\
\mu _{4} & \mu _{3} & m_{2}%
\end{array}%
\right) \left(
\begin{array}{ccc}
l_{4} & l_{3} & l \\
0 & 0 & 0%
\end{array}%
\right) \left(
\begin{array}{ccc}
l_{4} & l_{2} & l \\
-\mu _{4} & -\mu _{2} & m_{2}%
\end{array}%
\right) \left(
\begin{array}{ccc}
l_{4} & l_{2} & l \\
0 & 0 & 0%
\end{array}%
\right) \frac{\delta _{\mu _{3}}^{\mu _{2}}\delta _{l_{2}}^{l_{3}}}{2l_{3}+1}
\end{equation*}%
\begin{equation*}
=\sum_{l_{1}l_{2}l_{3}l_{4}}C_{l_{1}}C_{l_{2}}^{2}C_{l_{4}}\left(
\begin{array}{ccc}
l_{1} & l_{2} & l \\
0 & 0 & 0%
\end{array}%
\right) ^{2}\left(
\begin{array}{ccc}
l_{4} & l_{2} & l \\
0 & 0 & 0%
\end{array}%
\right) ^{2}\frac{(2l_{1}+1)(2l_{2}+1)(2l_{4}+1)(2l+1)^{2}}{(4\pi )^{2}}%
\text{..}
\end{equation*}%
Likewise, for the second term we note that $(-1)^{\mu _{1}+\mu _{2}+\mu
_{3}+\mu _{4}}\equiv 1,$ and using (\ref{6j1})%
\begin{equation*}
\sum_{m_{1}m_{2}}(-1)^{m_{1}+m_{2}}B(m_{1},-m_{1},m_{2},-m_{2})
\end{equation*}%
\begin{equation*}
=\sum_{l_{1}l_{2}l_{3}l_{4}}C_{l_{1}}C_{l_{2}}C_{l_{3}}C_{l_{4}}\left\{
\begin{array}{ccc}
l_{1} & l_{3} & l \\
l_{4} & l_{2} & l%
\end{array}%
\right\} \left(
\begin{array}{ccc}
l_{1} & l_{2} & l \\
0 & 0 & 0%
\end{array}%
\right)\times
\end{equation*}
\begin{equation*}
\times \left(
\begin{array}{ccc}
l_{3} & l_{4} & l \\
0 & 0 & 0%
\end{array}%
\right) \left(
\begin{array}{ccc}
l_{1} & l_{3} & l \\
0 & 0 & 0%
\end{array}%
\right) \left(
\begin{array}{ccc}
l_{4} & l_{2} & l \\
0 & 0 & 0%
\end{array}%
\right) \frac{(2l+1)^{2}\prod\limits_{i=1}^{4}(2l_{i}+1)}{(4\pi )^{2}}\text{
.}
\end{equation*}%
Now by Cauchy-Schwartz inequality and recalling that%
\begin{equation*}
\left\vert \left\{
\begin{array}{ccc}
l_{1} & l_{3} & l \\
l_{2} & l_{4} & l%
\end{array}%
\right\} \right\vert \leq \frac{1}{2l+1}\text{ for all }%
l_{1},l_{2},l_{3},l_{4}\text{ ,}
\end{equation*}%
the previous quantity can be bounded by%
\begin{equation*}
\frac{1}{2l+1}%
\sum_{l_{1}l_{2}l_{3}l_{4}}C_{l_{1}}C_{l_{2}}C_{l_{3}}C_{l_{4}}\left(
\begin{array}{ccc}
l_{1} & l_{2} & l \\
0 & 0 & 0%
\end{array}%
\right) \left(
\begin{array}{ccc}
l_{3} & l_{4} & l \\
0 & 0 & 0%
\end{array}%
\right) \left(
\begin{array}{ccc}
l_{1} & l_{3} & l \\
0 & 0 & 0%
\end{array}%
\right) \left(
\begin{array}{ccc}
l_{4} & l_{2} & l \\
0 & 0 & 0%
\end{array}%
\right) \frac{(2l+1)^{2}\prod\limits_{i=1}^{4}(2l_{i}+1)}{(4\pi )^{2}}
\end{equation*}%
\begin{equation*}
\leq \left[ \sum_{l_{1}l_{2}l_{3}l_{4}}C_{l_{1}}C_{l_{2}}C_{l_{3}}C_{l_{4}}%
\left(
\begin{array}{ccc}
l_{1} & l_{2} & l \\
0 & 0 & 0%
\end{array}%
\right) ^{2}\left(
\begin{array}{ccc}
l_{3} & l_{4} & l \\
0 & 0 & 0%
\end{array}%
\right) ^{2}\frac{(2l_{3}+1)(2l_{4}+1)}{4\pi }\frac{%
(2l_{1}+1)(2l_{2}+1)(2l+1)}{4\pi }\right] ^{1/2}
\end{equation*}%
\begin{equation*}
\times \left[
\sum_{l_{1}l_{2}l_{3}l_{4}}C_{l_{1}}C_{l_{2}}C_{l_{3}}C_{l_{4}}\left(
\begin{array}{ccc}
l_{1} & l_{3} & l \\
0 & 0 & 0%
\end{array}%
\right) ^{2}\left(
\begin{array}{ccc}
l_{2} & l_{4} & l \\
0 & 0 & 0%
\end{array}%
\right) ^{2}\frac{(2l_{3}+1)(2l_{4}+1)}{4\pi }\frac{%
(2l_{1}+1)(2l_{2}+1)(2l+1)}{4\pi }\right] ^{1/2}
\end{equation*}%
\begin{equation*}
=\frac{2l+1}{4}C_{l;2}^{2}\text{ ,}
\end{equation*}%
whence%
\begin{equation*}
\left\vert \frac{8}{(2l+1)^{2}C_{l;2}^{2}}%
\sum_{m_{1},m_{2}}B(m_{1},-m_{1},m_{2},-m_{2})\right\vert \leq \frac{2}{2l+1}%
\text{ .}
\end{equation*}%
It is easy to see that $\sum_{m_{1}m_{2}}A(m_{1},-m_{1},m_{2},-m_{2})=%
\sum_{m_{1}m_{2}}C(m_{1},-m_{1},m_{2},-m_{2}).$ In view of \ref{sum0}, \ref%
{sum00}, the statement of the lemma follows easily.
\end{proof}

The proof of Theorem \ref{T : MAINquad} is now concluded by the following
lemma.

\begin{lemma}
If $\beta =0$ and $\alpha >2$, then%
\begin{equation*}
\lim \inf_{l\rightarrow \infty }E\left\{ \widetilde{C}_{l;2}-1\right\}
^{2}\geq C_{2}^{2}\left\{ \frac{c_{2}^{3}}{2c_{1}^{2}\pi }\left\{ 2\zeta
(\alpha -1)+\zeta (\alpha )\right\} 2^{\alpha }\right\} ^{-2}>0\text{.}
\end{equation*}%
If $\beta >0$ and $\alpha $ is real, then%
\begin{equation*}
\lim_{l\rightarrow \infty }E\left\{ \widetilde{C}_{l;2}-1\right\} ^{2}=0%
\text{ .}
\end{equation*}
\end{lemma}

\begin{proof}
For the first part, from Lemma \ref{sabato} we can focus on%
\begin{equation*}
\frac{1}{C_{l;2}^{2}}\sum_{l_{1}l_{2}l_{3}}C_{l_{1}}^{2}C_{l_{2}}C_{l_{3}}%
\left(
\begin{array}{ccc}
l_{1} & l_{2} & l \\
0 & 0 & 0%
\end{array}%
\right) ^{2}\left(
\begin{array}{ccc}
l_{1} & l_{3} & l \\
0 & 0 & 0%
\end{array}%
\right) ^{2}\frac{(2l_{1}+1)(2l_{2}+1)(2l_{3}+1)}{(4\pi )^{2}}
\end{equation*}%
\begin{equation*}
=\frac{1}{C_{l;2}^{2}}%
\sum_{l_{1}l_{2}}(2l_{1}+1)(2l_{2}+1)C_{l_{1}}C_{l_{2}}\left(
\begin{array}{ccc}
l_{1} & l_{2} & l \\
0 & 0 & 0%
\end{array}%
\right) ^{2}\sum_{l_{3}}C_{l_{1}}C_{l_{3}}\left(
\begin{array}{ccc}
l_{1} & l_{3} & l \\
0 & 0 & 0%
\end{array}%
\right) ^{2}\frac{(2l_{3}+1)}{(4\pi )^{2}}\text{ ,}
\end{equation*}%
which is larger than%
\begin{equation*}
\frac{1}{C_{l;2}^{2}}\sum_{l_{2}}(2l_{2}+1)C_{2}C_{l_{2}}\left(
\begin{array}{ccc}
2 & l_{2} & l \\
0 & 0 & 0%
\end{array}%
\right) ^{2}\sum_{l_{3}}C_{2}C_{l_{3}}\left(
\begin{array}{ccc}
2 & l_{3} & l \\
0 & 0 & 0%
\end{array}%
\right) ^{2}\frac{(2l_{3}+1)}{(4\pi )^{2}}
\end{equation*}
\begin{equation*}
\geq \frac{C_{2}^{2}C_{l+2}^{2}}{C_{l;2}^{2}}\sum_{l_{2}}(2l_{2}+1)\left(
\begin{array}{ccc}
2 & l_{2} & l \\
0 & 0 & 0%
\end{array}%
\right) ^{2}\sum_{l_{3}}\left(
\begin{array}{ccc}
2 & l_{3} & l \\
0 & 0 & 0%
\end{array}%
\right) ^{2}\frac{(2l_{3}+1)}{(4\pi )^{2}}=\frac{C_{2}^{2}C_{l+2}^{2}}{%
C_{l;2}^{2}}\text{ .}
\end{equation*}%
Now we have proved earlier that in the polynomial case, $C_{l;2}\simeq
C_{l}\simeq l^{-\alpha },$ so the previous ratio does not converge to zero
and $\widehat{C}_{l;2}$ cannot be ergodic; the lower bound provided in the
statement of the Lemma follows from previous computations and easy
manipulations.

For the second part of the statement, it is sufficient to note that
\begin{eqnarray*}
&&\frac{1}{C_{l;2}^{2}}%
\sum_{l_{1}l_{2}}(2l_{1}+1)(2l_{2}+1)C_{l_{1}}C_{l_{2}}\left(
\begin{array}{ccc}
l_{1} & l_{2} & l \\
0 & 0 & 0%
\end{array}%
\right) ^{2}\sum_{l_{3}}C_{l_{1}}C_{l_{3}}\left(
\begin{array}{ccc}
l_{1} & l_{3} & l \\
0 & 0 & 0%
\end{array}%
\right) ^{2}\frac{(2l_{3}+1)}{(4\pi )^{2}} \\
&\leq &\frac{\sup_{l_{1}}(2l_{1}+1)^{-1}\sum_{l_{3}}\Gamma _{l_{1}}\Gamma
_{l_{3}}\left\{ C_{l_{1}0l_{3}0}^{l0}\right\} ^{2}}{\sum_{l_{1}l_{3}}\Gamma
_{l_{1}}\Gamma _{l_{3}}\left\{ C_{l_{1}0l_{3}0}^{l0}\right\} ^{2}}\leq \frac{%
\sup_{l_{1}}\sum_{l_{3}}\Gamma _{l_{1}}\Gamma _{l_{3}}\left\{
C_{l_{1}0l_{3}0}^{l0}\right\} ^{2}}{\sum_{l_{1}l_{3}}\Gamma _{l_{1}}\Gamma
_{l_{3}}\left\{ C_{l_{1}0l_{3}0}^{l0}\right\} ^{2}}\text{ ,}
\end{eqnarray*}%
so the condition is met, just as for the standard case of \cite{MPBER}.
\end{proof}

\bigskip

\noindent \textbf{Remarks. }(1)\textbf{\ }By inspection of the previous
proof, we note that we have shown how the sufficient condition for
asymptotic Gaussianity (HFG) is also such for ergodicity (HFE). More
precisely, we have proved that
\begin{equation*}
\lim_{l\rightarrow \infty }\sup_{l_{1}}\frac{\sup_{l_{1}}\sum_{l}\Gamma
_{l_{1}}\Gamma _{l_{2}}\left\{ C_{l_{1}0l_{2}0}^{l0}\right\} ^{2}}{%
\sum_{l_{1}l_{2}}\Gamma _{l_{1}}\Gamma _{l_{2}}\left\{
C_{l_{1}0l_{2}0}^{l0}\right\} ^{2}}=\lim_{l\rightarrow \infty }\sup_{\lambda
}P(Z_{1}=l_{1}%
{\vert}%
Z_{2}=l_{2})=0\text{ ,}
\end{equation*}%
where $\left\{ Z_{l}\right\} $ is the Markov chain defined in \cite[Eq. (57)
and (58)]{MPBER} is a sufficient condition for the HFG (see %
\cite[Proposition 9]{MPBER}) and also a sufficient condition to have $%
\lim_{l\rightarrow \infty }E\left\{ \widetilde{C}_{l}-1\right\} ^{2}=0$ .

(2) In principle, the case $q=3$ can be dealt along similar lines.

(3) (\emph{On Cosmic Variance}) Loosely speaking, the epistemological status
of Cosmological research has always been the object of some debate, as in
some sense we are dealing with a science based on a single observation (our
observed Universe). In the CMB community, this issue has been somewhat
rephrased in terms of so-called \emph{Cosmic Variance - }i.e., it taken as
common knowledge that parameters relating only to lower multipoles (such as
the value of $C_{l}$, for small values of $l$) are inevitably affected by an
intrinsic uncertainty which cannot be eliminated (the variability due to the
peculiar realization of the random field that we are able to observe),
whereas this effect is taken to disappear at higher $l$ (implicitly assuming
that something like the $HFE$ should always hold). Our result seem to point
out, apparently for the first time, the very profound role that the
assumption of Gaussianity may play in this environment. In particular, for
general non-Gaussian fields there is no guarantee that angular power spectra
and related parameters can be consistently estimated, even at high
multipoles - i.e., the Cosmic Variance does not decrease at high frequencies
for general non-Gaussian models.

\section{Discussion and directions for further research}

This paper leaves many directions open for further research. We believe the
results of the previous two sections point out a very strong connection
between conditions for High Frequency Ergodicity (HFE) and High Frequency
Gaussianity (HFG) for isotropic spherical random fields. It is natural to
suggest that equivalence may hold for Gaussian subordinated fields of any
order $q,$ or even more broadly for general Gaussian subordinated fields on
homogeneous spaces of compact groups. Indeed, in this broader framework it
is shown in \cite{BaMaVa} that independence of Fourier coefficients implies
Gaussianity, which is the heuristic rationale behind our results here.

The connection between the HFE property can also be studied under a
different environment than Gaussian subordination. Consider for instance the
class of completely random spherical fields, which was recently introduced
in \cite{dennis04,dennis05}. Following the definition therein, we shall say
that a spherical random field is completely random if for each $l$ we have
that the vector $a_{l.}=(a_{l,-l},...,a_{ll})$ is invariant with respect to
the action of all matrices belonging to $SU(2l+1)$ and verifies $%
a_{lm}=(-1)^{m}\overline{a}_{lm}.$ Because of this, the vector $a_{l.}$ is
clearly uniformly distributed on the manifold of random diameter $\sum
|a_{lm}|^{2}=\widehat{C}_{l},$ or equivalently, introducing the $(2l+1)$
vector $U_{l}$
\begin{equation}
U_{l}=\frac{1}{\sqrt{2l+1}}\left\{ \frac{\sqrt{2}{\rm Re}\,a_{l1}}{\sqrt{%
\widehat{C}_{l}}},\frac{\sqrt{2}{\rm Re}\,a_{l2}}{\sqrt{\widehat{C}_{l}}}%
,....,\frac{a_{l0}}{\sqrt{\widehat{C}_{l}}},\frac{\sqrt{2}{\rm Im}\,a_{l1}}{%
\sqrt{\widehat{C}_{l}}},...,\frac{\sqrt{2}{\rm Im}\,a_{ll}}{\sqrt{\widehat{C}%
_{l}}}\right\}  \label{fut0}
\end{equation}%
it holds that, for $l$ large, it holds approximately that $U_{l}\overset{law}%
{\sim }U(S^{2l}),$ i.e. $U_{l}$ it is asymptotically distributed on the unit
sphere of $\mathbb{R}^{2l+1}.$ Under these conditions, it is simple to show
that $HFE\Rightarrow HFG,$ i.e.
\begin{equation*}
\left\{ \lim_{l\rightarrow \infty }E\left\{ \widetilde{C}_{l}-1\right\}
^{2}=0\right\} \Rightarrow \left\{ \frac{T_{l}(x)}{\sqrt{Var(T_{l})}}\overset%
{law}{\rightarrow }N(0,1)\text{ , as }l\rightarrow \infty \right\} \text{ .}
\end{equation*}%
Indeed, it is sufficient to note that, as before
\begin{equation*}
\frac{T_{l}(x)}{\sqrt{Var(T_{l})}}=\frac{T_{l}}{\sqrt{(2l+1)C_{l}}}\overset{%
law}{=}\frac{\sqrt{4\pi }a_{l0}}{\sqrt{C_{l}}}\text{ ,}
\end{equation*}%
which we can write as
\begin{equation*}
\frac{a_{l0}}{\sqrt{C_{l}}}=\frac{a_{l0}}{\sqrt{\widehat{C}_{l}}}\sqrt{\frac{%
\widehat{C}_{l}}{C_{l}}}=\frac{a_{l0}}{\sqrt{\widehat{C}_{l}}}\sqrt{%
\widetilde{C}_{l}}\text{ .}
\end{equation*}%
Now, as $l\rightarrow \infty $
\begin{equation*}
\frac{a_{l0}}{\sqrt{\widehat{C}_{l}}}\overset{law}{\rightarrow }N(0,1)\text{
,}
\end{equation*}%
because the left hand side can be viewed as the marginal distribution for a
uniform law on a sphere of growing dimension; the latter is asymptotically
Gaussian, as a consequence of Poincar\'{e} Lemma (see \cite{diadif}). We do
not investigate this issue more fully here, and we leave for future research
the determination of general conditions such that (compare with (\ref{fut0}))%
\begin{equation}
\text{the law of }U_{l}\text{ and }U(S^{2l})\text{ are asymptotically close
as }l\rightarrow \infty \text{ .}  \label{fut}
\end{equation}

Obviously, for all fields such that (\ref{fut}) holds (i.e. those that are
asymptotically completely random, to mimic the terminology of \cite%
{dennis04,dennis05}), by the same argument as before we have that
\begin{equation*}
\left\{ \sqrt{\frac{\widehat{C}_{l}}{C_{l}}}\rightarrow _{prob}1\right\}
\Rightarrow \left\{ \frac{T_{l}}{\sqrt{(2l+1)C_{l}}}\overset{law}{%
\rightarrow }N(0,1)\right\} \text{ . }
\end{equation*}%
To conclude this work, we wish to provide an example were the HFE and HFG
property are indeed not equivalent. Consider the (anisotropic) field%
\begin{equation*}
h(x)=\sum_{lm}\xi _{lm}Y_{lm}(x)\text{ , where }\xi _{lm}=\left\{
\begin{array}{c}
\xi _{l}\text{ , for }m=0 \\
0\text{ , otherwise}%
\end{array}%
\right. ,
\end{equation*}%
and the random variables $\xi _{l}$ verifies the assumption%
\begin{equation*}
E\xi _{l}=0\text{, }\sum_{l}E\xi _{l}^{2}<\infty \text{ \ \ and \ }E\xi
_{l}^{4}<\infty \text{ .}
\end{equation*}%
Note that, in the definition of $h\left( x\right) $, the sum is not taken
with respect to $l$. The field can be made isotropic by taking a random
rotation $T(x)=h(gx),$ where $g$ is a random, uniformly distributed element
of $SO(3)$. We have as usual $T(x)=\sum_{l}\sum_{m=-l}^{l}a_{lm}Y_{lm}(x)$ ,
where%
\begin{equation*}
a_{lm}\overset{law}{=}\sum_{m^{\prime }=-l}^{l}D_{m^{\prime }m}^{l}(g)\xi
_{lm^{\prime }}\overset{law}{=}\sqrt{\frac{4\pi }{2l+1}}Y_{lm}(g)\xi _{l}%
\text{ ,}
\end{equation*}%
and where $\left\{ D^{l}(g)\right\} $ denotes the well-known Wigner
representation matrices for $SO(3),$ and the first identity in law is
discussed for instance in \cite{BaMa}, \cite{MPJMVA}. Note that%
\begin{equation*}
\sum_{m=-l}^{l}|a_{lm}|^{2}=\frac{4\pi }{2l+1}\xi
_{l}^{2}\sum_{m=-l}^{l}\left\vert Y_{lm}(g)\right\vert ^{2}=\xi _{l}^{2}%
\text{ ,}
\end{equation*}%
as expected, because the sample angular power spectrum is invariant to
rotations. Of course in this case we do not have ergodicity in general, i.e.
it may happen that%
\begin{equation*}
\frac{\sum_{m=-l}^{l}|a_{lm}|^{2}}{E\sum_{m=-l}^{l}|a_{lm}|^{2}}=\frac{\xi
_{l}^{2}}{E\xi _{l}^{2}}\nrightarrow 1
\end{equation*}%
and indeed for general sequences $\left\{ \xi _{l}\right\} $
\begin{equation*}
E\left\{ \frac{\xi _{l}^{2}}{E\xi _{l}^{2}}-1\right\} ^{2}=E\left\{ \frac{%
\xi _{l}^{2}}{E\xi _{l}^{2}}\right\} ^{2}-1\neq 0\text{ .}
\end{equation*}%
However, in the special case where%
\begin{equation*}
\xi _{l}=\left\{
\begin{array}{l}
e^{-l}\text{ with probability }\frac{1}{2} \\
-e^{-l}\text{ with probability }\frac{1}{2}%
\end{array}%
\right. \text{ ,}
\end{equation*}%
we obtain easily that $E\left\{ \frac{\xi _{l}^{2}}{E\xi _{l}^{2}}-1\right\}
^{2}\equiv 0,$ while asymptotic Gaussianity fails. Hence, we have
constructed an example where the HFE property holds but the HFG property
does not. Note that the support of the vector $\left\{ a_{l.}\right\} $ is
concentrated on a small subset of the sphere $S^{2l}$; heuristically, this
is what prevents Poincar\`{e}-like arguments to go through.

\end{document}